\documentclass[]{article}
\usepackage[utf8]{inputenc}
\usepackage[style=alphabetic, backend=bibtex]{biblatex}
\usepackage{amsmath}
\usepackage{amsfonts}
\usepackage{amsthm}
\usepackage{amssymb}
\usepackage{framed}
\usepackage{tkz-graph}
\usepackage{tikz}
\usetikzlibrary{arrows}
\usepackage{enumerate}
\usepackage{mathtools}
\usepackage{algpseudocode}
\usepackage{framed}
\usepackage{mdframed}
\usepackage{todonotes}
\usepackage{xfrac}
\usepackage{caption}
\usepackage{subcaption}

\newtheorem{theorem}{Theorem}
\newtheorem{corollary}[theorem]{Corollary}

\newtheorem{lemma}[theorem]{Lemma}
\newtheorem{convention}[theorem]{Convention}

\theoremstyle{definition}
\newtheorem{definition}[theorem]{Definition}

\newtheorem{observation}[theorem]{Observation}
\newtheorem{case}{Case}

\makeatletter
\@addtoreset{case}{lemma}
\makeatother

\newcommand{\tw}{\operatorname{tw}}
\newcommand{\ord}{\operatorname{ord}}

\title{Minimum Cycle Decomposition: A Constructive Characterization for Graphs of Treewidth Two with Node Degrees Two and Four}
\author{Irene Heinrich and Sven O.~Krumke}

\begin{document}

\maketitle


\begin{abstract}
	Substantial efforts have been made to compute or estimate the minimum number $c(G)$ of cycles needed to partition the edges of an Eulerian graph.
	We give an equivalent characterization of Eulerian graphs of treewidth $2$ and with maximum degree $4$.
	This characterization enables us to present a linear time algorithm for the computation of $c(G)$ for all $G$ in this class.
\end{abstract}

\section{Introduction}

\subsection{Overview and Preliminary Results}
It is a well-known result that a graph is Eulerian if and only if it is connected and it decomposes into cycles.
In this article we focus on finding a minimum cycle decomposition for a graph $G$, i.e., a decomposition into a minimum number $c(G)$ of cycles.

There has been a lot of research on an upper bound for $c(G)$.
Already in 1968, Hajós conjectured that if $G$ is a simple graph, it holds that $c(G) \leq \lfloor \frac{n}{2} \rfloor$, cf.\ \cite{Lovasz68}. Dean showed in \cite{DeanDicycleDecomposition} that the upper bound can be replaced by $ \lfloor \frac{n-1}{2} \rfloor$. Granville and Moisiadis have shown in \cite{GranMoiDeg4} that Hajós conjecture holds true for Eulerian simple graphs of maximum degree four. Seyffarth showed that it is true for every planar simple graph, cf.\ \cite{SeyffPlanar}.
Moreover, Fan and Xu showed in \cite{Fan_Xu02} that Hajós conjecture also holds true for all projective planar graphs and all graphs which do not contain $K_6-$ as a minor.
The general conjectue is still open today.

Our goal is to present the class of double ear decomposable graphs on which the exact computation of $c(G)$ is possible in linear time. This class is defined through a constructive characterization: We start with an arbitrary cycle. Then, in every step, we can either subdivide the graph or add a so-called double ear. Every graph that can by constructed in a finite number of such steps is called a \emph{double ear decomposable graph}.

We further show that a graph is double ear decomposable if and only if it is of treewidth at most two and has only node degrees two or four. Thus, we can give a constructive characterization of these graphs. An overview on constructive characterizations in graph theory can be found in \cite{KovacConstrChar}.

Finally, we present an algorithm that decides in linear time whether a graph is double ear decomposable. In this case, the algorithm also outputs a double ear decomposition. Furthermore, it calculates a minimum cycle decomposition in linear time.

\subsection{Notation and Terminology}
The graphs discussed in this article may contain multiple edges or loops.
Let $G = (V, E)$ be a graph. We set $\ord(G):= |V|$. Let $v \in V$ and $e \in E$.
Then $G-e$ denotes the graph obtained by removing the edge $e$ from $G$ and $G - v$ denotes the graph obtained from removing $v$ and all of its adjacent edges from $G$.
A \emph{cut} of $G$ is a partition $V = V_1 \cup V_2$ of $V$ such that none of the sets $V_1$ and $V_2$ is empty.
An edge $e$ is \emph{contained} in the cut if it connects a node from $V_1$ with a node from $V_2$.
The \emph{weight} of a cut is the number of edges it contains.
We denote with $\Delta(G)$ the maximum node degree in $G$.
An \emph{even graph} is a graph containing only nodes of even degree.
A \emph{decomposition} of a graph is a set of subgraphs such that each edge appears in exactly one subgraph in the set.


\begin{definition}[Cycle Decomposition]
	Let $G$ be an even graph. A decomposition of $G$ into cycles is called a \emph{cycle decomposition}.
	We call
	\[c(G) = \min\{|\mathcal{C}| \mid \mathcal{C} \text{ is a cycle decomposition of } G\}\]
	the \emph{cycle number} of $G$.
\end{definition}

\section{Cycles in Double Ear Decomposable Graphs}

\begin{observation}
	Let $G$ be an even graph. If $G_1, G_2, \dots, G_l$. are the connected components of $G$, then
	\[c(G) = \sum_{i=1}^l c(G_i).\]
\end{observation}

Unless stated otherwise, we only consider connected graphs from now on. The even connected graphs are exactly the Eulerian graphs.

\begin{observation}
	\label{obs: 2cut Eulerian}
	Let $G$ be an Eulerian graph.
	If there is a cut in $G$ of weight two containing the edges $e$ and $e'$, every cycle decomposition of $G$ contains a cycle $C$ with $e, e' \in E(C)$.
\end{observation}

\begin{definition}[Subdivision, Resolving a Node of Degree $2$]
	A graph $G'$ that is obtained by replacing an edge with endnodes $u$ and $w$ of a graph $G$ by a new path $uvw$ is called a \emph{simple subdivision} of $G$.
	We say that $G$ is obtained from $G'$ by \emph{resolving} the node $v$.
	Moreover, $G''$ is called a \emph{subdivision} of $G$ if there is a finite sequence $(G_0 = G, G_1, G_2, \dots, G_k = G'')$ such that $G_i$ is a simple subdivision of $G_{i-1}$ for each $i = 1, \dots, k$.
\end{definition}

\begin{figure}[htbp]
	\caption{Subdividing the triple edge into a house by simple subdivision steps: The green lines are exchanged by length-two-paths in the next step.}
	\centering
	\begin{subfigure}[t]{0.2\textwidth}
		\centering
		\begin{tikzpicture}
		\clip(-0.1,-0.9) rectangle (1.1,0.8);
		\node (a) at (0,0) {$a$};
		\node (b) at (1,0) {$b$};
		\draw[-] (a) edge (b);
		\draw[-] (a) edge [bend left = 20, green] (b);
		\draw[-] (a) edge [bend right = 20] (b);
		\end{tikzpicture}
	\end{subfigure}
	\begin{subfigure}[t]{0.2\textwidth}
		\centering
		\begin{tikzpicture}
		\clip(-0.1,-0.9) rectangle (1.1,0.8);
		\node (a) at (0,0) {$a$};
		\node (b) at (1,0) {$b$};
		\node (c) at (0.5,0.7) {$c$};
		\draw[-] (a) edge (b);
		\draw[-] (a) edge [bend right = 20, green] (b);
		\draw[-] (a) edge (c);
		\draw[-] (b) edge (c);
		\end{tikzpicture}
	\end{subfigure}
	\begin{subfigure}[t]{0.2\textwidth}
		\centering
		\begin{tikzpicture}
		\clip(-0.1,-0.9) rectangle (1.1,0.8);
		\node (a) at (0,0) {$a$};
		\node (b) at (1,0) {$b$};
		\node (c) at (0.5,0.7) {$c$};
		\node (d) at (0,-0.7) {$d$};
		\draw[-] (a) edge (b);
		\draw[-] (a) edge (c);
		\draw[-] (b) edge (c);
		\draw[-] (a) edge (d);
		\draw[-] (b) edge [green] (d);
		\end{tikzpicture}
	\end{subfigure}
	\begin{subfigure}[t]{0.2\textwidth}
		\centering
		\begin{tikzpicture}
		\clip(-0.1,-0.9) rectangle (1.1,0.8);
		\node (a) at (0,0) {$a$};
		\node (b) at (1,0) {$b$};
		\node (c) at (0.5,0.7) {$c$};
		\node (d) at (0,-0.7) {$d$};
		\node (e) at (1,-0.7) {$e$};
		\draw[-] (a) edge (b);
		\draw[-] (a) edge (c);
		\draw[-] (b) edge (c);
		\draw[-] (a) edge (d);
		\draw[-] (b) edge (e);
		\draw[-] (d) edge (e);
		\end{tikzpicture}
	\end{subfigure}
\end{figure}
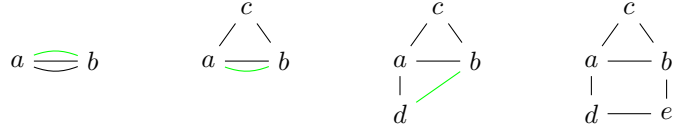

\begin{observation}
	\label{obs: subdivision same cycles}
	If $G'$ is a subdivision of $G$, then $c(G) = c(G')$.
\end{observation}

\begin{observation}
	\label{obs: articulation cycles}
	Let $e_1, \dots, e_l$ be the different loops of a graph $G$. Then \[c(G) = l + c(G - \{e_1, \dots, e_l\}).\]
\end{observation}


\subsection{Double Ear Decompositions}
\begin{definition}[Adding a Double Ear]
	\label{def: Adding a Double Ear}
	Let $G$ be a graph containing a path $P = v_1 \cdots v_r$ such that $\forall i = 1, \dots, r: \deg_G(v_i) = 2$.
	We \emph{add a double ear to $P$} by the following construction:
	\begin{enumerate}
		\item \label{step: double ear connect} Connect the first and the last vertex of $P$ by a new edge.
		\item \label{step: double ear duplicate} Duplicate each edge of $P$.
	\end{enumerate}
	Also the path $v_1$ without edges is allowed as a choice of $P$. In this case, a loop is added to $v_1$.
	Observe that all edges added in the above two steps build a cycle. In particular: If $G$ is an Eulerian graph, also the graph obtained by adding a double ear is an Eulerian graph.
\end{definition}

\begin{definition}
	Let $(G_0, G_1, \dots, G_k)$ be a finite sequence of graphs such that for each $i = 1, \dots, k$: either
	\begin{itemize}
		\item \emph{double ear step:} $G_i$ is either built by adding a double ear to $G_{i-1}$ or
		\item \emph{subdivision step:} $G_i$ is a simple subdivision of $G_{i-1}$.
	\end{itemize}
	Then, we call $(G_0, G_1, \dots, G_k)$ a double ear decomposition of $G_k$ starting with $G_0$.
	If $G_0$ is a cycle, we call $G_k$ a \emph{double ear decomposable graph}.
\end{definition}

\subsection{Double Ears Correspond to a Minimum Cycle Decomposition}
\begin{theorem}
	\label{theorem: ded cycle number}
	Let $(G_0, G_1, \dots, G_{k'} = G)$ be a double ear decomposition of a graph $G$ with $k \leq k'$ double ear steps.
	Then
	\[c(G) = c(G_0) + k.\]
\end{theorem}

\begin{proof}
	We prove the above statement by induction on $k$.
	If $k = 0$, then $G$ is a subdivision of $G_0$ and thus $c(G) = c(G_0) + 0 = c(G_0) + k$ by Observation \ref{obs: subdivision same cycles}.
	Otherwise let $k \geq 1$.
	Let $1 \leq i \leq k'$ such that $G_i$ is the last double ear step of the decomposition $(G_0, G_1, \dots, G_{k'})$.
	Then, $G$ is a subdivision of $G_i$ and therefore $c(G) = c(G_i)$. By induction $c(G_{i-1}) = c(G_0) + k-1 $ It remains to show that $c(G_i) = c(G_{i-1}) + 1$.
	
	There is a degree-two-path $P = v_0 v_1 \cdots v_r$ in $G_{i-1}$ such that $G_{i}$ is $G_{i-1}$ plus a double ear added to $P$.
	\begin{case} [$v_0 = v_r$] Then, $P$ is a path of length zero in $G_{i-1}$. 
		This means, $G_i$ is $G_{i-1}$ with one loop added at $v_0$. 
		We obtain $c(G_{i}) = c(G_{i-1}) + 1$ by Observation \ref{obs: articulation cycles}.
	\end{case}
	\begin{case}[$v_0 \neq v_r$]
	There is exactly one edge $e_0$ in $G_{i-1}$ outside of $P$ which is incident to $v_0$. Analogously let $e_{r+1}$ be the unique edge in $G_{i-1}$ outside $P$ and incident to $v_r$.
	Let $G^A$ and $G^B$ be the two connected components of $G_i- e_0 - e_{r+1}$. We may assume that $P \subseteq G^B$.
	The problem of finding a minimum cycle decomposition of $G_i$ directly corresponds to solving both of the following two problems:
	Let $a(e_0)$ and $a(e_{r+1})$ be the endpoints of $e_0$ and $e_{r+1}$ in $G^A$. Let $b(e_0)$ and $b(e_{r+1})$ be the endpoints of $e_0$ and $e_{r+1}$ in $G^B$.
	\begin{enumerate}[(A)]
		\item \label{prob A} Decompose $G^A$ into cycles and one path with endnodes $a(e_0)$ and $a(e_{r+1})$ such that the number of cycles in the decomposition is minimized.
		\item \label{prob B} Decompose $G^B$ into cycles and one path with endnodes $b(e_0)$ and $b(e_{r+1})$ such that the number of cycles in the decomposition is minimized.
	\end{enumerate}
	By Observation \ref{obs: 2cut Eulerian}, a minimum cycle decomposition of $G_i$ clearly induces an optimal solution for \ref{prob A} and \ref{prob B}.
	Moreover, if $(C_1^{\ref{prob A}}, C_2^{\ref{prob A}}, \dots, C_{k_{\ref{prob A}}}^{\ref{prob A}}, P^{\ref{prob A}})$ is an optimal solution for \ref{prob A} and
	$(C_1^{\ref{prob B}}, C_2^{\ref{prob B}}, \dots, C_{k_{\ref{prob B}}}^{\ref{prob B}}, P^{\ref{prob B}})$ is an optimal solution for \ref{prob B},
	where $P^{\ref{prob B}}$ and $P^{\ref{prob B}}$ are the respective paths, then
	\[(C_1^{\ref{prob A}}, C_2^{\ref{prob A}}, \dots, C_{k_{\ref{prob A}}}^{\ref{prob A}}, C_1^{\ref{prob B}}, C_2^{\ref{prob B}}, \dots, C_{k_{\ref{prob B}}}^{\ref{prob B}},  P^{\ref{prob A}} + P^{\ref{prob B}} + e_0 + e_{r+1})\] is a minimum cycle decomposition of $G_i$ by Observation \ref{obs: 2cut Eulerian}.
	We can now solve the two problems separately to obtain a minimum cycle decomposition for $G_i$:
	\begin{enumerate}[(A)]
		\item Any minimum cycle decomposition $\mathcal{C} = \{C_1, C_2 \dots, C_{k-1}\}$ for $G_{i-1}$ induces an optimal solution $\{C_1, C_2, \dots, C_{k-2}, P^{\ref{prob A}}\}$ for Problem $\ref{prob A}$.
		\item $\{P, C\}$ is an optimal solution for Problem \ref{prob B}, where $C$ is the cycle consisting of all edges in $E(G_i)\setminus E(G_{i-1})$, i.\ e.\ $C$ is the double ear added to $G_{i-1}$.
		The solution is optimal since $G^B$ is not just a path. So at least one cycle is needed for the desired decomposition.
	\end{enumerate}
	A minimum cycle decomposition for $G_i$ now is
	\begin{align*}
		\{C_1, C_2, \dots, C_{k-2}, C,  P^{\ref{prob A}} + P + e_0 + e_{r+1}\} = \{C_1, C_2, \dots, C_{k-2}, C_{k-1}, C\}.
	\end{align*}
	Altogether $c(G) = c(G_i) = c(G_{i-1}) + 1 = k$.\qedhere
	\end{case}
	\setcounter{case}{0}
\end{proof}

As an immediate consequence we obtain the following corollary:
\begin{corollary}
	Let $G$ be a double ear decomposable graph with a double ear decomposition $(G_0, G_1, \dots, G_{k'} = G)$ containing $k \leq k'$ double ear steps.
	Then $c(G) = k+1$.
\end{corollary}

\section{A Constructive Characterization of Eulerian Graphs of Treewidth Less or Equal to Two and with Degrees Less or Equal to Four}
The main goal of this section is to prove the following constructive characterization.

\begin{theorem}
	\label{theorem: EquivalentDoubleEarDecomposition}
	A graph $G$ is double ear decomposable if and only if $\tw(G) \in \{1,2\}$ and its node degrees lie in $\{2,4\}$.
\end{theorem}

\begin{definition}[Tree Decomposition, Treewidth, Bags, \cite{BodArboretum}]
	A \emph{tree decomposition} of a graph $G = (V, E)$ is a pair $(\{X_i\colon i \in I\}, T=(I,F))$ with $\{X_i\colon i \in I\}$ a family of subsets of $V$, one for each node of $T$, and $T$ a tree such that
	\begin{itemize}
		\item $\bigcup_{i \in I} X_i = V$.
		\item For all edges $vw \in E$, there exists an $i \in I$ with $v \in X_i$ and $w \in X_i$.
		\item For all $i,j,k \in I$: If $j$ is on the path from $i$ to $k$ in $T$, then $X_i \cap X_k \subseteq X_j$.
	\end{itemize}
	For $i \in I$, $X_i$ is called a \emph{bag} of the tree decomposition.
	The \emph{width} of a tree decomposition $(\{X_i\colon i \in I\}, T=(I,F))$ is $\max_{i \in I} |X_i|-1$.
	The \emph{treewidth} of a graph $G$ is the minimum width over all possible tree decompositions of $G$.
\end{definition}

\begin{lemma}[\cite{BodArboretum}]
	\label{lemma: SmoothTreeDecomposition}
	Suppose the treewidth of $G$ is $k$. $G$ has a tree decomposition
	$\left(\left\{ X_i, i \in I \right\}, T = (I, F)\right)$ of width $k$ such that
	\begin{itemize}
		\item For all $i \in I\colon |X_i| = k+1$.
		\item For all $(i,j) \in F\colon |X_i \cap X_j| = k$.
	\end{itemize}
\end{lemma}

\begin{definition}
	A tree decomposition as in Lemma \ref{lemma: SmoothTreeDecomposition} is called a \emph{smooth} tree decomposition.
\end{definition}

\begin{lemma}
	\label{lemma: LeavesInSmoothTreeDecompositions}
	Let $\left(\left\{ X_i, i \in I \right\}, T = (I, F)\right)$ be a smooth tree decomposition of a graph $G$.
	Let $i$ be a leave of $T$. Then, there exists a vertex $v \in X_i$ whose neighbourhood is a subset of $X_i$.
\end{lemma}

\begin{proof}
	If $i$ is the only node of $T$, then there is nothing to show.
	Otherwise, let $j$ be the unique neighbour of $i$ in $T$. By assumption $|X_i \cap X_j| = k$.
	Thus, there is exactly one node $v$ in $X_i$ which does not lie in $X_j$.
	This node must not appear in any other node of $T$ since every path between $i$ and another node of $T$ traverses $j$.
	Thus, all neighbours of $v$ are contained in $X_i$.
\end{proof}

\begin{lemma}
	\label{lemma: treewidth subdivision}
	Let $G'$ be a subdivision of a graph $G$. Then
	\[\tw(G) \leq 2 \Leftrightarrow \tw(G') \leq 2.\]
\end{lemma}

\begin{proof}
	The graph $G$ is a minor of $G'$. Thus, $\tw(G) \leq \tw(G')$.
	
	Let now $G$ be a graph with $\tw(G) \leq 2$.
	We show that subdividing one edge does not increase the treewidth to more than $2$.
	The statement above then follows inductively. Let $xz$ be the edge of $G$ which is replaced by the edges $xy$ and $yz$ in the subdivision.
	Let $(X,T)$ be a minimum tree decomposition of $G$ with $X = \{X_1, \dots, X_l\}$. Then, there must be a bag $X_i \in X$ containing $x$ and $z$. We add a new bag $X_{l+1} = \{x,y,z\}$ to $T$ and a new edge $X_{l+1}X_i$ to $T$. The result is a tree decomposition of $G'$ of width $2$.
\end{proof}

\begin{observation}
	\label{obs: parallel edges treewidth}
	Adding parallel edges does not change the treewidth. The same tree decomposition can be maintained.
\end{observation}

We are now able to show that all double ear decomposable graphs are of treewidth less than or equal to two and that they have all node degrees in $\{2,4\}$.

\begin{proof}[Proof of Theorem \ref{theorem: EquivalentDoubleEarDecomposition}, Part I]
	Let $G$ be a double ear decomposable graph with a decomposition $(G_0, G_1, \dots, G_k)$.
	It follows immediately by induction on $k$ that $\tw(G) \leq 2$ (cf.\ Figure \ref{fig: doubleEarDecomposableGraphsAreOfTWLeq2}) and for all $v \in V(G)\colon \deg_G(v) 
	\in \{2,4\}$.
	\begin{figure}[htbp]
		\caption{Adding a double-ear does not change the treewidth.}
		\label{fig: doubleEarDecomposableGraphsAreOfTWLeq2}
		\centering
		\begin{subfigure}[]{0.3\textwidth}
			\centering
			\begin{tikzpicture}[yscale = 0.5, xscale = 0.8]
			\node (a) at (1,1) {$a$};
			\node (b) at (2,1.6) {$b$};
			\node (c) at (2.7,2.4) {$c$};
			\node (d) at (2.7,3.5) {$d$};
			\node (e) at (2,4.3) {$e$};
			\node (f) at (1,4.9) {$f$};
			\draw[-] (a) edge (b);
			\draw[-] (b) edge (c);
			\draw[-] (c) edge (d);
			\draw[-] (d) edge (e);
			\draw[-] (e) edge (f);
			\draw[dashed] (f) -- (0.2, 4.9);
			\draw[dashed] (a) -- (0.2, 1);
			\end{tikzpicture}
			\caption*{(I) The treewidth of this graph is less than or equal to two (by induction).}
		\end{subfigure}\qquad
		\begin{subfigure}[h]{0.3\textwidth}
			\centering
			\begin{tikzpicture}[yscale = 0.5, xscale = 0.8]
				\node (a) at (1,1) {$a$};
				\node (b) at (2,1.6) {$b$};
				\node (e) at (2,4.3) {$e$};
				\node (f) at (1,4.9) {$f$};
				\draw[-] (b) edge (e);
				\draw[-] (a) edge  (b);
				\draw (e) -- (f);
				\draw[dashed] (f) -- (0.2, 4.9);
				\draw[dashed] (a) -- (0.2, 1);
			\end{tikzpicture}
			\caption*{(II) By Lemma \ref{lemma: treewidth subdivision} the treewidth of this graph is less than or equal to two.}
		\end{subfigure}\qquad
		\begin{subfigure}[h]{0.28\textwidth}
			\centering
			\begin{tikzpicture}[yscale = 0.5, xscale = 0.8]
			\node (a) at (1,1) {$a$};
			\node (b) at (2,1.6) {$b$};
			\node (e) at (2,4.3) {$e$};
			\node (f) at (1,4.9) {$f$};
			\draw[-] (b) edge [bend right = 10] (e);
			\draw[-] (b) edge [orange, bend left = 10] (e);
			\draw[-] (a) edge  (b);
			\draw (e) -- (f);
			\draw[dashed] (f) -- (0.2, 4.9);
			\draw[dashed] (a) -- (0.2, 1);
			\end{tikzpicture}
			\caption*{(III) Adding a \textcolor{orange}{parallel edge} does not change the treewidth.}
		\end{subfigure}
		\begin{subfigure}[h]{0.45\textwidth}
			\centering
			\begin{tikzpicture}[yscale = 0.5, xscale = 0.8]
			\node (a) at (1,1) {$a$};
			\node (b) at (2,1.6) {$b$};
			\node (c) at (2.7,2.4) {$c$};
			\node (d) at (2.7,3.5) {$d$};
			\node (e) at (2,4.3) {$e$};
			\node (f) at (1,4.9) {$f$};
			\draw[-] (a) edge (b);
			\draw[-] (b) edge (c);
			\draw[-] (c) edge (d);
			\draw[-] (d) edge (e);
			\draw[-] (e) edge (f);
			\draw[-] (b) edge [orange] (e);
			\draw[dashed] (f) -- (0.2, 4.9);
			\draw[dashed] (a) -- (0.2, 1);
			\end{tikzpicture}
			\caption*{(IV) Subdividing the edge $be$ does not change the treewidth by Lemma \ref{lemma: treewidth subdivision}}	
		\end{subfigure}\qquad
		\begin{subfigure}[h]{0.45\textwidth}
			\centering
			\begin{tikzpicture}[yscale = 0.5, xscale = 0.8]
			\node (a) at (1,1) {$a$};
			\node (b) at (2,1.6) {$b$};
			\node (c) at (2.7,2.4) {$c$};
			\node (d) at (2.7,3.5) {$d$};
			\node (e) at (2,4.3) {$e$};
			\node (f) at (1,4.9) {$f$};
			\draw[-] (a) edge (b);
			\draw[-] (b) edge [bend right = 15] (c);
			\draw[-] (c) edge [bend right = 8] (d);
			\draw[-] (d) edge [bend right = 15] (e);
			\draw[-] (e) edge (f);
			\draw[-] (b) edge [orange] (e);
			\draw[-] (b) edge [red, bend left = 15] (c);
			\draw[-] (c) edge [red, bend left = 8] (d);
			\draw[-] (d) edge [red, bend left = 15] (e);
			\draw[dashed] (f) -- (0.2, 4.9);
			\draw[dashed] (a) -- (0.2, 1);
			\end{tikzpicture}
			\caption*{(V) Again, adding \textcolor{red}{parallel edges} does not change the treewidth.}	
		\end{subfigure}\qquad
	\end{figure}
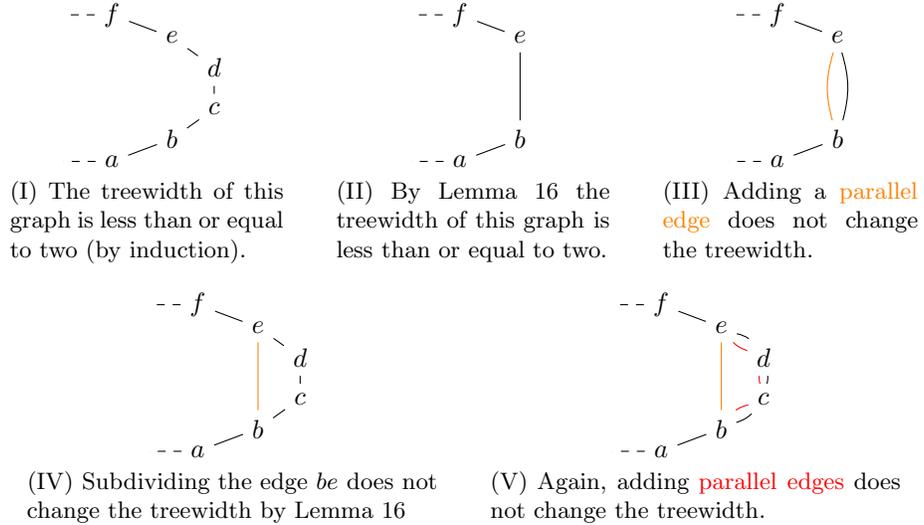
\end{proof}

\begin{lemma}
	\label{lemma: timeI}
	Let $(G_0, G_1, \dots, G_{k-1}, G)$ be a double ear decomposition of $G$ starting with a cycle.
	Assume that there are exactly two parallel edges $e_1$ and $e_2$ connecting the nodes $v$ and $w$ in $G$.
	Then, exactly one of the following statements holds true:
	\begin{enumerate}[a)]
		\item $G_0 = G$ consists of the nodes $v, w$ and the two parallel edges $e_1$ and $e_2$.
		\item There is an index $0 \leq i \leq k-1$ such that in $G_i$ there is a loop wich is subdivided into the two edges $e_1$ and $e_2$ in $G_{i+1}$.
		\item One of the parallel edges arises from a double ear step: There is an $0 \leq i \leq k-1$ such that one of the two edges (W.l.o.g.\ $e_1$) is part of $G_i$ whereas $e_2$ appears firstly in $G_{i+1}$. Furthermore, $G_{i+1}$ is $G_i$ plus a double ear added to some path $P$ and
		\begin{itemize}
			\item either $G_i$ is a cycle, $e_1 \notin P$ and $v,w$ are the two ends of $P$
			\item or $e_1 \in P$ and at least one of the nodes $v,w$ lies in the interior of $P$.
		\end{itemize}	
	\end{enumerate}
\end{lemma}

\begin{proof}
	We may assume that the index where $e_1$ appears first in a graph $G_i, i = 0, \dots, k$ is less than or equal to the index $i^*$ where $e_2$ appears first. This enhances us to distinguish the following three cases. Let 
	\begin{case}[$i^* = 0$] This means that $e_1$ and $e_2$ are already contained in $G_0$. Since $G_0$ is a cycle, we get that the two parallel edges $e_1$ and $e_2$ form the whole cycle $G_0$.
	\end{case}
	\begin{case}[$i^* > 0$ and $G_{i^* -1} \rightarrow G_{i^*}$ is a simple subdivision]
	The only possibility for a parallel edge to arise from a simple subdivision is the subdivision of a loop.
	\end{case}
	\begin{case}[$i^* > 0$ and $G_{i^* -1} \rightarrow G_{i^*}$ is a double ear step]
	 Let $P$ be the degree two path in $G_{i^* -1}$	to which the double ear is added. If $P$ contains $e_1$, then at least one end of $e_1$ is in the interior of $P$. Otherwise, the double ear step would result in a triple edge which contradicts the assumption. If $P$ does not contain $e_1$, its both ends must be $v$ and $w$. Consequently $G_{i^* -1}$ is a cycle since all nodes in $P$ are of degree two.
	\end{case}
	\setcounter{case}{0}
\end{proof}

\begin{lemma}
	\label{lemma: timeII}
	Let $(G_0, G_1, \dots, G_{k-1}, G)$ be a double ear decomposition starting with a cycle.
	Assume that there are exactly three parallel edges $e_1, e_2$ and $e_3$ connecting two nodes $v$ and $w$ in $G$.
	Then, there is an index $0 \leq i \leq k-1$ such that $\deg_{G_i}(u) = \deg_{G_i}(v) = 2$, there is an edge $e$ in $G_i$ connecting $v$ with $w$
	and $G_{i+1}$ is constructed by adding a double ear to the path $vw$ in $G_i$.
\end{lemma}
\begin{proof}
	A triple edge can not be generated by a simple subdivision and neither can $G_0$ be a triple edge.
	Thus, there must be an index $i^*$ such that $G_{i^*-1} \rightarrow G_{i^*}$ is a double ear step and $G_{i^*}$ is the first graph in the decomposition that contains all three parallel edges $e_1, e_2, e_3$.
	Consequently, exactly one of the three edges, say $e_1$, already is contained in the edge set of $G_{i^* - 1}$ and $G_{i^*}$ is constructed by adding a double ear to the path $vw$.
\end{proof}

\begin{lemma}
	\label{lemma: timeIII}
	Let $(G_0, G_1, \dots, G_{k-1}, G)$ be an ear decomposition starting with a cycle.
	Let $e_1, e_2, e_3$ and $e_4$ be four parallel edges in $G$ connecting two different nodes $v, w \in V(G)$.
	Then, $G$ is a quadruple edge.
\end{lemma}
\begin{proof}
	From theorem \ref{theorem: EquivalentDoubleEarDecomposition}, we know that the maximum degree of a node in $G$ is four. Thus, $G$ is a quadruple edge.
\end{proof}

\begin{proof}[Proof of Theorem \ref{theorem: EquivalentDoubleEarDecomposition}, Part II]

	Let $G$ be a graph with $\deg(v) \in \{2,4\}$ for all $v \in V(G)$ and $\tw(G) \in \{1,2\}$.
	We prove that $G$ is a double ear decomposable graph by induction on the number $n$ of vertices in $G$.
	
	In the case of $n = 1$, $G$ is either a single or a double loop. Both are double ear decomposable.
	If $G$ is a single loop, $\{G\}$ is a double ear decomposition since $G$ is a cycle itself.
	Otherwise $\{G_0, G\}$ is a double ear decomposition, where $G_0$ is a single loop. The double ear step is done by considering the length-zero-path consisting of the only node in $G_0$. 
	
	Let now $n \geq 2$. By Lemma \ref{lemma: SmoothTreeDecomposition} there is a smooth tree decomposition $\left(\left\{ X_i, i \in I \right\}, T = (I, F)\right)$ of $G$ of width $2$.
	From Lemma \ref{lemma: LeavesInSmoothTreeDecompositions}, we obtain a node $u \in V$ that appears in exactly one bag $X_i$ and all of the neighbours of $u$ are contained in $X_i$.
	Thus, $u$ has either exactly one neighbour $v$ or it has two neighbours $v$ and $w$.
	
	The main concept behind this induction is to consider a graph $G'$ (mostly) with node set $V - u$ and $\tw(G') \leq 2$ and node degrees $2$ and $4$. By induction, $G'$ has a double ear decomposition. The idea then is to manipulate this decomposition of $G'$ such that we obtain a decomposition of $G$.
	\begin{case}[$u$ has one neighbour $v$ and $\deg(u) = 2$]
		\label{case: degu2 1 neighbour}
		If $\deg(v) = 2$ then $G$ is a cycle and thus $(G)$ is a double ear decomposition of $G$. Otherwise, we consider $G' = G-u$.
		Then $\tw(G') \leq 2$ and all of the degrees in $G'$ are still two or four. Thus, we can apply the induction hypothesis to obtain that $G'$ is double ear decomposable.
		Observe that $\deg_{G'}(v) = 2$.
		We choose the path $v$ of length zero for a next double ear step and obtain a graph $G''$ which is $G'$ with a loop added to $v$. We subdivide this loop to obtain $G$.
	\end{case}
	\begin{case}[$u$ has one neighbour $v$, $\deg(u) = 4$ and there is no loop at $u$]
		\label{case: degu4 1 neighbour}
		The graph $G$ is a quadruple edge by Lemma \ref{lemma: timeIII}.
		We find the double ear decomposition of $G$ as follows:
		Start with a cycle of length two and then add a double ear by considering a path of length one.
	\end{case}
	\begin{case}[$u$ has two neighbours $v,w$ and $\deg_G(u) = 2$]
		\label{case: degu2 2 neighbours}
		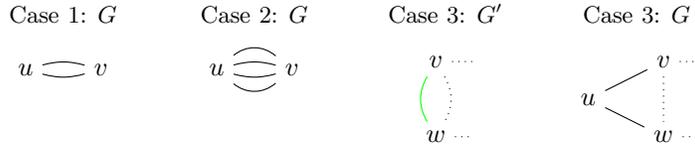
\begin{figure}[htbp]
			\centering
			\caption{First three cases.}
			\begin{subfigure}[t]{0.2\textwidth}
				\caption*{Case \ref{case: degu2 1 neighbour}: $G$}
				\centering
				\begin{tikzpicture}
				\clip(0.9,-0.3) rectangle (2.1,0.3);
				\node (u) at (1,0) {$u$};
				\node (v) at (2,0) {$v$};
				\draw[-] (u) edge [bend right = 15] (v);
				\draw[-] (u) edge [bend left = 15] (v);
				\end{tikzpicture}
			\end{subfigure}
			\begin{subfigure}[t]{0.2\textwidth}
				\caption*{Case \ref{case: degu4 1 neighbour}: $G$}
				\centering
				\begin{tikzpicture}
				\clip(0.9,-0.3) rectangle (2.1,0.3);
				\node (u) at (1,0) {$u$};
				\node (v) at (2,0) {$v$};
				\draw[-] (u) edge [bend right = 40] (v);
				\draw[-] (u) edge [bend right = 15] (v);
				\draw[-] (u) edge [bend left = 15] (v);
				\draw[-] (u) edge [bend left = 40] (v);
				\end{tikzpicture}
			\end{subfigure}
			\begin{subfigure}[t]{0.2\textwidth}
				\caption*{Case \ref{case: degu2 2 neighbours}: $G'$}
				\centering
				\begin{tikzpicture}
				\node (v) at (2,0.5) {$v$};
				\node (w) at (2,-0.5) {$w$};
				\draw[-] (v) edge [bend right, green] (w);
				\draw[-] (v) edge [bend left, dotted] (w);
				\draw[dotted] (v) -- (2.5,0.5);
				\draw[dotted] (w) -- (2.5,-0.5);
				\end{tikzpicture}
			\end{subfigure}
			\begin{subfigure}[t]{0.2\textwidth}
				\caption*{Case \ref{case: degu2 2 neighbours}: $G$}
				\centering
				\begin{tikzpicture}
				\node (u) at (1,0) {$u$};
				\node (v) at (2,0.5) {$v$};
				\node (w) at (2,-0.5) {$w$};
				\draw (u) -- (v);
				\draw (u) -- (w);
				\draw[dotted] (v) -- (w);
				\draw[dotted] (v) -- (2.5,0.5);
				\draw[dotted] (w) -- (2.5,-0.5);
				\end{tikzpicture}
			\end{subfigure}
		\end{figure}
		Construct a graph $G'$ by removing $u$ and adding an edge between $v$ and $w$.
		All degrees in $G'$ are preserved from $G$ and $\tw(G') \leq 2$. By induction we obtain that $G'$ has a double ear decomposition.
		We construct $G$ by subdividing an edge between $v$ and $w$ in $G'$. Thus, also $G$ has a double ear decomposition.
	\end{case}
	\begin{case}[$u$ shares one edge with $v$ and three edges with $w$]
		\label{case: udeg4 1v}
		Construct $G'$ by removing $u$ and adding one edge between $v$ and $w$. Observe that $\deg_{G'}(w) = 2$.
		By induction $G'$ is a double ear decomposable graph with decomposition $(G_0, G_1, \dots, G')$.
		We construct a graph $G''$ by subdividing the edge between $v$ and $w$ with $u$. Now, $G$ is obtained from $G''$ by adding a double ear to the path $vu$.
		Thus, $(G_0, G_1, \dots, G'', G)$ is a double ear decomposition of $G$.
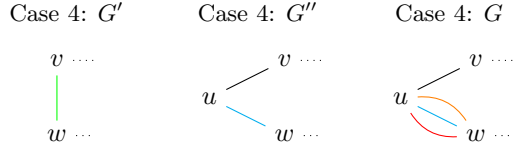
		\begin{figure}[htbp]
			\centering
				\begin{subfigure}[h]{0.2\textwidth}
					\caption*{Case \ref{case: udeg4 1v}: $G'$}
					\centering
					\begin{tikzpicture}
					\node (v) at (2,0.5) {$v$};
					\node (w) at (2,-0.5) {$w$};
					\draw[-] (v) edge [green] (w);
					\draw[dotted] (v) -- (2.5,0.5);
					\draw[dotted] (w) -- (2.5,-0.5);
					\end{tikzpicture}
				\end{subfigure}	
				\begin{subfigure}[h]{0.2\textwidth}
					\caption*{Case \ref{case: udeg4 1v}: $G''$}
					\centering
					\begin{tikzpicture}
					\node (u) at (1,0) {$u$};
					\node (v) at (2,0.5) {$v$};
					\node (w) at (2,-0.5) {$w$};
					\draw (u) -- (v);
					\draw[-] (u) edge [cyan] (w);
					\draw[dotted] (v) -- (2.5,0.5);
					\draw[dotted] (w) -- (2.5,-0.5);
					\end{tikzpicture}
				\end{subfigure}
			\begin{subfigure}[h]{0.2\textwidth}
				\caption*{Case \ref{case: udeg4 1v}: $G$}
				\centering
				\begin{tikzpicture}
				\node (u) at (1,0) {$u$};
				\node (v) at (2,0.5) {$v$};
				\node (w) at (2,-0.5) {$w$};
				\draw (u) -- (v);
				\draw[-] (u) edge [cyan] (w);
				\draw[-] (u) edge [bend left, orange] (w);
				\draw[-] (u) edge [bend right, red] (w);
				\draw[dotted] (v) -- (2.5,0.5);
				\draw[dotted] (w) -- (2.5,-0.5);
				\end{tikzpicture}
			\end{subfigure}
		\end{figure}
	\end{case}
	\begin{case}[$u$ shares two edges with $v$ and two edges with $w$, the nodes $v$ and $w$ are adjacent]
		\label{case: udeg4 2v2w}
		Thus, $\deg_G(v) = \deg_G(w) = 4$.
		Assume that $u,v$ and $w$ are the only nodes of $G$.
		Then, each pair of nodes is connected by two exactly two edges in $G$.
		Set $G_0 = (\{u,v,w\}, \{uv,vw,wu\})$. Then $(G_0, G)$ is a double ear decomposition of $G$.
		
		Otherwise, there exists at least one other node in $G$. Let $x$ be the additional neighbour of $v$ and $y$ be the additional neighbour of $w$. Note that $x$ and $y$ might be the same node.
		Construct $G'$ by removing $u,v$ and $w$ from $G$ and adding one edge between $x$ and $y$.
		By induction $G'$ has a double ear decomposition $(G_0, G_1, \dots, G')$.
		Construct a graph $G''$ via subdivision of the edge between $x$ and $y$ by inserting $v,u,w$.
		Insert an ear at the path $wuv$. The resulting graph is $G$.
		Altogether $G$ has a double ear decomposition $(G_0, G_1, \dots, G', G'', G)$.
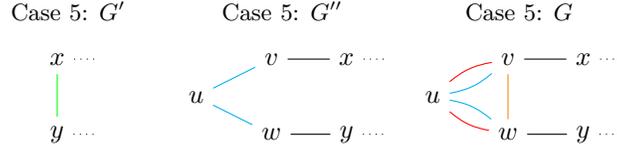
		\begin{figure}[htbp]
			\centering
			\begin{subfigure}[h]{0.2\textwidth}
				\caption*{Case \ref{case: udeg4 2v2w}: $G'$}
				\centering
				\begin{tikzpicture}
				\node (x) at (3,0.5) {$x$};
				\node (y) at (3,-0.5) {$y$};
				\draw[-] (x) edge [green] (y);
				\draw[dotted] (x) -- (3.5,0.5);
				\draw[dotted] (y) -- (3.5,-0.5);
				\end{tikzpicture}
			\end{subfigure}	
			\begin{subfigure}[h]{0.25\textwidth}
				\caption*{Case \ref{case: udeg4 2v2w}: $G''$}
				\centering
				\begin{tikzpicture}
				\node (u) at (1,0) {$u$};
				\node (v) at (2,0.5) {$v$};
				\node (w) at (2,-0.5) {$w$};
				\node (x) at (3,0.5) {$x$};
				\node (y) at (3,-0.5) {$y$};
				\draw[-] (u) edge [cyan] (v);
				\draw[-] (u) edge [cyan] (w);
				\draw (v) -- (x);
				\draw (w) -- (y);
				\draw[dotted] (x) -- (3.5,0.5);
				\draw[dotted] (y) -- (3.5,-0.5);
				\end{tikzpicture}
			\end{subfigure}
			\begin{subfigure}[h]{0.25\textwidth}
				\caption*{Case \ref{case: udeg4 2v2w}: $G$}
				\centering
				\begin{tikzpicture}
				\node (u) at (1,0) {$u$};
				\node (v) at (2,0.5) {$v$};
				\node (w) at (2,-0.5) {$w$};
				\node (x) at (3,0.5) {$x$};
				\node (y) at (3,-0.5) {$y$};
				\draw[-] (v) edge [orange] (w);
				\draw[-] (u) edge [bend left = 15, red] (v);
				\draw[-] (u) edge [bend right = 15, cyan] (v);
				\draw[-] (u) edge [bend left = 15, cyan] (w);
				\draw[-] (u) edge [bend right = 15, red] (w);
				\draw (v) -- (x);
				\draw (w) -- (y);
				\draw[dotted] (x) -- (3.5,0.5);
				\draw[dotted] (y) -- (3.5,-0.5);
				\end{tikzpicture}
			\end{subfigure}
		\end{figure}
	\end{case}
	\begin{case}[$u$ shares two edges with $v$ and two edges with $w$, the nodes $v$ and $w$ are not adjacent, $\deg_G(w) = \deg_G(v) = 4$]
		\label{case udeg4 2v2w v not adjacent to w}
		Construct $G'$ by removing $u$ and adding two edges, say $e_1$ and $e_2$, between $v$ and $w$.
		By induction, $G'$ has a double ear decomposition $(G_0, G_1, \dots, G_k = G')$.
		In $G'$, there are exaclty two edges connecting $v$ and $w$.
		Hence, we can apply Lemma \ref{lemma: timeI} and obtain three possible cases:
		\begin{enumerate}[a)]
			\item $G'$ is isomorphic to a double edge: This case can not occur since $v$ and $w$ are of degree $4$.
			\item There is an index $0 \leq i \leq k-1$ such that in $G_i$ there is a loop $e$ which is subdivided into the two edges $e_1$ and $e_2$ to construct $G_{i+1}$: W.l.o.g.\ let $v$ be the node that subdivides the loop.
			Construct $G_i'$ by subdividing $e$ with the node $u$.
			Now, consider the length-zero-path $u$ in $G_i'$ and add a double ear (i.e.\ a loop $e'$) to it. The resulting graph is called $G_i''$.
			Construct a third graph $G_{i+1}'$ by subdividing $e'$ with $v$.
			
			We obtain a double ear decomposition
			\[(G_1, \dots, G_i, G_i', G_i'', G_{i+1}', G_{i+2}', G_{i+1}', \dots, G_k' = G)\] of $G$,
			where for all $j \geq 1\colon$ $G_{j+1}'$ arises from $G_{j}'$ by applying the same construction step that is used to construct $G_{j+1}$ from $G_j$.
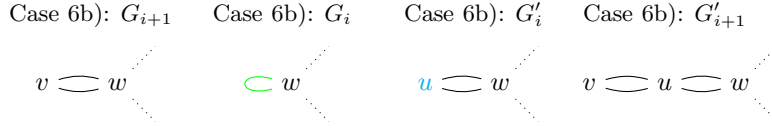
					\begin{figure}[htbp]
				\centering
				\begin{subfigure}[h]{0.2\textwidth}
					\caption*{Case \ref{case udeg4 2v2w v not adjacent to w}b): $G_{i+1}$}
					\centering
					\begin{tikzpicture}[every loop/.style={}]
					\node (v) at (2,0) {$w$};
					\node (u) at (1,0) {$v$};
					\draw[dotted] (v) -- (2.5,0.5);
					\draw[dotted] (v) -- (2.5,-0.5);
					\draw[-] (u) edge [bend left = 15] (v);
					\draw[-] (u) edge [bend right = 15] (v);
					\end{tikzpicture}
				\end{subfigure}	
				\begin{subfigure}[h]{0.2\textwidth}
					\caption*{Case \ref{case udeg4 2v2w v not adjacent to w}b): $G_i$}
					\centering
					\begin{tikzpicture}[every loop/.style={}]
					\node[] (v) at (2,0) {$w$};
					\draw[dotted] (v) -- (2.5,0.5);
					\draw[dotted] (v) -- (2.5,-0.5);
					\draw[-] (v) edge [loop left, green] (v);
					\end{tikzpicture}
				\end{subfigure}
				\begin{subfigure}[h]{0.2\textwidth}
					\caption*{Case \ref{case udeg4 2v2w v not adjacent to w}b): $G_i'$}
					\centering
					\begin{tikzpicture}[every loop/.style={}]
					\node (v) at (2,0) {$w$};
					\node[cyan] (u) at (1,0) {$u$};
					\draw[dotted] (v) -- (2.5,0.5);
					\draw[dotted] (v) -- (2.5,-0.5);
					\draw[-] (u) edge [bend left = 15] (v);
					\draw[-] (u) edge [bend right = 15] (v);
					\end{tikzpicture}
				\end{subfigure}
				\begin{subfigure}[h]{0.2\textwidth}
					\caption*{Case \ref{case udeg4 2v2w v not adjacent to w}b): $G_{i+1}'$}
					\centering
					\begin{tikzpicture}[every loop/.style={}]
					\node (u) at (1,0) {$u$};
					\node (v) at (2,0) {$w$};
					\node (w) at (0,0) {$v$};
					\draw[-] (u) edge [bend left = 15] (v);
					\draw[-] (u) edge [bend right = 15] (v);
					\draw[-] (u) edge [bend left = 15] (w);
					\draw[-] (u) edge [bend right = 15] (w);
					\draw[dotted] (v) -- (2.5,0.5);
					\draw[dotted] (v) -- (2.5,-0.5);
					\end{tikzpicture}
				\end{subfigure}
			\end{figure}
			
			\item
			One of the parallel edges arises from a double ear step: There is an $0 \leq i \leq k-1$ such that one of the two edges (W.l.o.g. $e_1$) is also an edge of $G_i$ whereas $e_2$ appears firstly in $G_{i+1}$. Furthermore, $G_{i+1}$ is $G_i$ plus a double ear added to some path $P$ and
			\begin{itemize}
				\item either $G_i$ is a cycle, say a $C_n$, $e_1 \notin P$ and $v,w$ are the two ends of $P$: In this case, we construct $G_i'$ by subdividing the edge $vw$ with $u$.
				Further, construct $G_{i+1}'$ by adding a double ear to the unique path between $u$ and $v$ that does not contain the edge between $u$ and $v$.
				Then, $(G_0, G_1, \dots, G_i, G_i', G_{i+1}', \dots, G_k' = G)$ is a double ear decomposition of $G$ where for all $j \geq 1\colon$ $G_{j+1}'$ arises from $G_{j}'$ by applying the same construction step that is used to construct $G_{j+1}$ from $G_j$.
				\item or $e_1 \in P$ and at least one of the nodes $v,w$ lies in the interior of $P$: We construct $G_i'$ by a simple subdivison of the edge $e_1$ in $G_i$. Let $P'$ be the path which is $P$ with the subdivided edge $e_1$. We construct $G_{i+1}'$ by adding a double ear to $P'$. From now on, we can use the same construction steps as in the double ear decomposition of $G'$ to obtain $G$.
			\end{itemize}
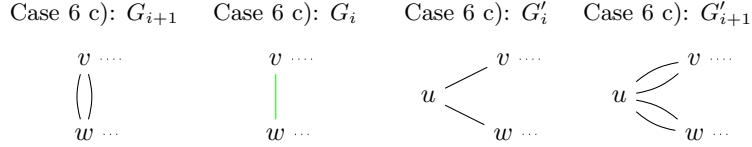
		\begin{figure}[htbp]
			\centering  
			\begin{subfigure}[h]{0.2\textwidth}
				\caption*{Case \ref{case udeg4 2v2w v not adjacent to w} c): $G_{i+1}$}
				\centering
				\begin{tikzpicture}
				\node (v) at (2,0.5) {$v$};
				\node (w) at (2,-0.5) {$w$};
				\draw[-] (v) edge [bend left = 15] (w);
				\draw[-] (v) edge [bend right = 15] (w);
				\draw[dotted] (v) -- (2.5,0.5);
				\draw[dotted] (w) -- (2.5,-0.5);
				\end{tikzpicture}
			\end{subfigure}	
			\begin{subfigure}[h]{0.2\textwidth}
				\caption*{Case \ref{case udeg4 2v2w v not adjacent to w} c): $G_i$}
				\centering
				\begin{tikzpicture}
				\node (v) at (2,0.5) {$v$};
				\node (w) at (2,-0.5) {$w$};
				\draw[-] (v) edge [green] (w);
				\draw[dotted] (v) -- (2.5,0.5);
				\draw[dotted] (w) -- (2.5,-0.5);
				\end{tikzpicture}
			\end{subfigure}
			\begin{subfigure}[h]{0.2\textwidth}
				\caption*{Case \ref{case udeg4 2v2w v not adjacent to w} c): $G_i'$}
				\centering
				\begin{tikzpicture}
				\node (u) at (1,0) {$u$};
				\node (v) at (2,0.5) {$v$};
				\node (w) at (2,-0.5) {$w$};
				\draw[-] (u) edge (v);
				\draw[-] (u) edge (w);
				\draw[dotted] (v) -- (2.5,0.5);
				\draw[dotted] (w) -- (2.5,-0.5);
				\end{tikzpicture}
			\end{subfigure}
			\begin{subfigure}[h]{0.2\textwidth}
				\caption*{Case \ref{case udeg4 2v2w v not adjacent to w} c): $G_{i+1}'$}
				\centering
				\begin{tikzpicture}
				\node (u) at (1,0) {$u$};
				\node (v) at (2,0.5) {$v$};
				\node (w) at (2,-0.5) {$w$};
				\draw[-] (u) edge [bend left = 15] (v);
				\draw[-] (u) edge [bend right = 15] (v);
				\draw[-] (u) edge [bend left = 15] (w);
				\draw[-] (u) edge [bend right = 15] (w);
				\draw[dotted] (v) -- (2.5,0.5);
				\draw[dotted] (w) -- (2.5,-0.5);
				\end{tikzpicture}
			\end{subfigure}
		\end{figure}
		\end{enumerate}

	\end{case}
	\begin{case}[There is a loop at $u$ and $\deg_G(u) = 4$]\hfill
		\label{case loop udeg4}
		\begin{enumerate}[a)]
			\item Assume that $u$ has two different neighbours $v$ and $w$.
			Construct $G'$ by removing $u$ and adding an edge between $v$ and $w$.
			By induction $G'$ is double ear decomposable.
			Next, we use a simple subdivison of and edge between $v$ and $w$ to obtain a graph $G''$.
			We call the subdividing node $u$. If we now add a double ear to the path $u$ in $G''$, we obtain $G$.
			\item If $u$ has only one neighbour $v$ in $G$, the construction is analogous:
			Construct $G'$ by removing $u$ and adding a loop to $v$. Then subdivide the loop by a node $u$ and add a double ear to the path $u$.\qedhere
		\end{enumerate}	
	\end{case}
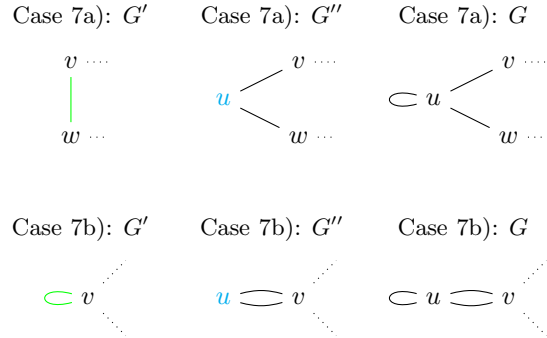
	\begin{figure}[htbp]
		\centering
		\begin{subfigure}[h]{0.2\textwidth}
			\caption*{Case \ref{case loop udeg4}a): $G'$}
			\centering
			\begin{tikzpicture}
			\node (v) at (2,0.5) {$v$};
			\node (w) at (2,-0.5) {$w$};
			\draw[-, green] (v) edge (w);
			\draw[dotted] (v) -- (2.5,0.5);
			\draw[dotted] (w) -- (2.5,-0.5);
			\end{tikzpicture}
		\end{subfigure}	
		\begin{subfigure}[h]{0.2\textwidth}
			\caption*{Case \ref{case loop udeg4}a): $G''$}
			\centering
			\begin{tikzpicture}
			\node[cyan] (u) at (1,0) {$u$};
			\node (v) at (2,0.5) {$v$};
			\node (w) at (2,-0.5) {$w$};
			\draw[-] (v) edge (u);
			\draw[-] (u) edge (w);
			\draw[dotted] (v) -- (2.5,0.5);
			\draw[dotted] (w) -- (2.5,-0.5);
			\end{tikzpicture}
		\end{subfigure}
		\begin{subfigure}[h]{0.2\textwidth}
			\caption*{Case \ref{case loop udeg4}a): $G$}
			\centering
			\begin{tikzpicture}[every loop/.style={}]
			\node (u) at (1,0) {$u$};
			\node (v) at (2,0.5) {$v$};
			\node (w) at (2,-0.5) {$w$};
			\draw[-] (u) edge (v);
			\draw[-] (u) edge (w);
			\draw[dotted] (v) -- (2.5,0.5);
			\draw[dotted] (w) -- (2.5,-0.5);
			\draw[-] (u) edge [loop left] (u);
			\end{tikzpicture}
		\end{subfigure}
	\end{figure}
\begin{figure}[htbp]
	\centering
	\begin{subfigure}[h]{0.2\textwidth}
		\caption*{Case \ref{case loop udeg4}b): $G'$}
		\centering
		\begin{tikzpicture}[every loop/.style={}]
		\node[] (v) at (2,0) {$v$};
		\draw[dotted] (v) -- (2.5,0.5);
		\draw[dotted] (v) -- (2.5,-0.5);
		\draw[-] (v) edge [loop left, green] (v);
		\end{tikzpicture}
	\end{subfigure}
	\begin{subfigure}[h]{0.2\textwidth}
		\caption*{Case \ref{case loop udeg4}b): $G''$}
		\centering
		\begin{tikzpicture}[every loop/.style={}]
		\node (v) at (2,0) {$v$};
		\node[cyan] (u) at (1,0) {$u$};
		\draw[dotted] (v) -- (2.5,0.5);
		\draw[dotted] (v) -- (2.5,-0.5);
		\draw[-] (u) edge [bend left = 15] (v);
		\draw[-] (u) edge [bend right = 15] (v);
		\end{tikzpicture}
	\end{subfigure}
	\begin{subfigure}[h]{0.2\textwidth}
		\caption*{Case \ref{case loop udeg4}b): $G$}
		\centering
		\begin{tikzpicture}[every loop/.style={}]
		\node (u) at (1,0) {$u$};
		\node (v) at (2,0) {$v$};
		\draw[-] (u) edge [bend left = 15] (v);
		\draw[-] (u) edge [bend right = 15] (v);
		\draw[dotted] (v) -- (2.5,0.5);
		\draw[dotted] (v) -- (2.5,-0.5);
		\draw[-] (u) edge [loop left] (u);
		\end{tikzpicture}
	\end{subfigure}
\end{figure}
\end{proof}

\section{Calculating Double Ear Decompositions}
Let $G$ be a graph.
Theorem \ref{theorem: EquivalentDoubleEarDecomposition} allows to decide in linear time, whether $G$ is double ear decomposable or not:
We check if all of the node degrees are two or four.
Moreover, $G$ is of treewidth two if and only if all of its biconnected components are series parallel graphs.
Calculating the biconnected components of $G$ is also possible in linear time as well as checking if the components are series parallel.

We will present another algorithm that tests if a graph $G$ is double ear decomposable (i.e.\ $\tw(G) \leq 2$ and all degrees are in $\{2,4\}$).
If this is the case, the algorithm outputs the minimum cycle number in linear time.

%

\subsection{Reduction Steps}
The idea of the algorithm is to apply the reduction steps below.

\subsubsection{Reduction I: Loops}
	If $e$ is a loop in $G$, and $e$ is not the only edge in $G$ set $G:= G-e$ and $c:= c+1$.
	In the decomposition a loop is added to the path in $G-e$ consisting only of the endnode of $e$.
	If $e$ is a loop and it is the only edge of $G$, the decomposition starts with the loop $e$.
	Return $c+1$ and the calculated decomposition.

\subsubsection{Reduction II: Double Edges}
	Assume $G$ contains two nodes $v$ and $w$ with $\deg(v) = \deg(w) = 4$ which are connected by exactly two edges $e_1'$ and $e_2'$.
	Calculate an inclusion-wise maximal path $P =(v_0, e_1, v_1, \dots, e_l, v_l)$ which fulfils
	\begin{itemize}
		\item $e_1' \in \{e_1, \dots, e_l\}$,
		\item for $i = 1, \dots, k-1$: $v_i$ has exactly two different neighbours.
	\end{itemize}
	Remove $P$ from $G$ and remove an edge connecting $v_0$ and $v_l$ from $G$.
	Set $c := c+1$.
	
%
%

\subsubsection{Reduction III: Triple Edges}
	Assume $G$ contains two nodes $v$ and $w$ with $\deg(v) = \deg(w) = 4$ which are connected by exactly three edges $e_1, e_2, e_3$.
	Set $G:= G- e_2 - e_3$ and $c:= c + 1$.

\subsubsection{Reduction IV: Quadruple Edges}
	If there is a quadruple edge, then $G$ is a quadruple edge by Lemma \ref{lemma: timeIII}.
	In this case the cycle decomposition of $G$ consists of two cycles of length two.
	Set $c:= c + 2$.
	The cycle decomposition of the quadruple edge starts with a cycle of length two.
	Then, a double ear is added to a path of length one to obtain the quadruple edge.
	Return $c$ and the double ear decomposition.
	
\subsubsection{Reduction V: Resolving}
	If $u$ is a node of degree two in $G$ and $G$ contains more that one node, we choose $G'$ to be the graph where $u$ is resolved.
	The value $c$ remains unchanged by Observation \ref{obs: subdivision same cycles}.
	
\subsubsection{Failure}
	If none of the above reduction steps is applicable, return "This is not a double ear decomposable graph!"

\subsection{The Algorithm and its Analysis}
\label{sec:algor-its-analys}

The algorithm has a simple structure: As long as one of the above reduction steps is possible, it applies this step.  To implement the algorithm, we first compute the degrees of all nodes and build two lists $V_2$ and $V_4$ of the nodes of degree~$2$ and~$4$, respectively.  We also replace each bundle~$B$ of parallel edges by a single edge with weight equal to the cardinality of~$B$ and build lists $E_2$, $E_3$ and $E_4$ of the edges with weights $2$, $3$ and $4$, respectively.  We also build a list~$L$ of loops in the graph. Each vertex and each edge is linked to its (potential) entries in the lists via a pointer.  We also add a marker $\mathtt{visited}$ to each vertex~$v\in V$ which initially is set to false.  The initial structures can be built in time linear in the size of the graph.

In each iteration we first test whether a Reduction~I which removes loops can be done.  This test can be done easily by inspecting the list~$L$ of loops in the graph.  If no loops exist, we then test whether a Reduction~V (resolving) is possible. This can be done in constant time by checking the list~$V_2$.  If $V_2$ is nonempty, say $v\in V_2$, and $v$ has two neighbors $u$ and~$w$, we check whether $w$ is already in the adjacency list of~$u$.  If this is not the case, we add $w$ to the adjacency list of~$u$ and vice versa. Otherwise we simply increase the weight of the edge~$uv$ by~$1$.  Finally, we remove~$v$ from the graph.  If necessary, we update the edge lists using the pointers attached to the edges.  Clearly, all of these steps can be carried out in constant time.

If no reductions of Type~$I$ and~$V$ are possible, we check whether a Reduction~IV is applicable.  This is straightforward by inspecting~$E_4$.  Notice that by Lemma~20, this reduction is the final reduction on a graph of constant size, so this final iteration needs constant time.

Observe that, given the situation that there are no quadruple edges, a Reduction~III is applicable if and only if the list~$E_3$ is nonempty.  Thus, we can handle this reduction by checking~$E_3$ and decreasing the weights of the involved edges as well as updating the vertex and edge lists in constant time. 

The final case that needs to be handled is the situation that the only potential reduction is one of Type~II.  We can now pick an arbitrary vertex~$u$ from the list~$V_4$ and one of its neighbors~$v$ of degree~$4$.  We mark both vertices~$u$ and~$v$ as visited by setting their marker $\mathtt{visited}$ to true and initialize the path to contain the edge~$uv$. We now start a search at~$u$ by setting~$u$ to be the active vertex.  If the current active vertex~$x$ has exactly two neighbors inductively, exactly one of those neighbors will be unvisited.  Let~$y$ be this neighbor, which then is connected to~$x$ by a double edge.  We add the edge~$xy$ to the path, mark~$y$ to be visited and make~$y$ the active vertex.  We carry out the analogous search from~$v$.  Since all degrees of the nodes in the graph are~$2$ or~$4$, this yields an inclusionwise maximal path~$P$ with the properties needed for Reduction~II.  Removing the path~$P$ from the graph, updating the degrees and lists and resetting all the touched markers as unvisited needs time proportional to the length of~$P$.  Thus, the overall running time of the algorithm is linear.
	
	\subsubsection{Correctness}
	We know by Theorem \ref{theorem: ded cycle number} that if $G$ is $G'$ plus some double ear, then $c(G) = c(G') + 1$.
	Thus, we only need to prove the following statement. The correctness then follows by induction.
	\begin{lemma}
		If $G$ is a double ear decomposable graph, then one of the above reduction steps is always possible and the reduced graph $G'$ is also double ear decomposable.
	\end{lemma} 

	\begin{proof}
		Let $(G_0, G_1, \dots, G_k = G)$ be a double ear decomposition of $G$. If the step from $G_{k-1}$ to $G$ is a subdivision, there is a node of degree $2$ in $G$ and we can apply the resolving-reduction (unless $G$ itself is a loop -- then we apply the loop reduction and the algorithm returns $c = 1$ and the double ear decomposition $(G)$).
		Otherwise $G$  is built by adding a double ear to  a degree-$2$-path $P$ in $G_{k-1}$. If $P$ is of length $0$, we can apply the loop-reduction. If $P$ is of length $1$, we can apply the triple- or the quadruple-edge-reduction. Otherwise, we can apply the double-edge-reduction.
		
		It remains to show that $G'$ is also a double ear decomposable graph.
		But this actually follows immediately with Theorem \ref{theorem: EquivalentDoubleEarDecomposition}. The treewidth is not increased by the reduction steps and also the degree condition stays true.
	\end{proof}
	
\printbibliography
\end{document}